\documentclass[twoside,centertags                 
]{amsart}

\usepackage[cmtip,color,all,
]{xy}

\usepackage{amssymb}
\usepackage{graphicx}
\usepackage{mathrsfs}

\newtheorem{thm}[equation]{Theorem}
\newtheorem{cor}[equation]{Corollary}
\newtheorem{lem}[equation]{Lemma}
\newtheorem{prop}[equation]{Proposition}
\theoremstyle{definition}

\theoremstyle{remark}
\newtheorem{rem}[equation]{Remark}

\newcommand{\C}[1]{\mathscr{#1}}

\newcommand{\D}[1]{\mathbb{#1}}

\newcommand{\der}{\mathbb{D}}
\newcommand{\perf}[1]{\operatorname{Perf}(#1)}
\newcommand{\db}{D^b}
\newcommand{\cb}{C^b}

\def\op{\mathrm{op}}

\def\r{\rightarrow} 

\newcommand{\deltan}[1]{\mathbf{#1}}

\newcommand{\Mod}[1]{\mathrm{Mod}(#1)}
\renewcommand{\mod}[1]{\mathrm{mod}(#1)}
\newcommand{\proj}[1]{\mathrm{proj}(#1)}
\newcommand{\sMod}[1]{\mathrm{\underline{Mod}}(#1)}
\newcommand{\smod}[1]{\mathrm{\underline{mod}}(#1)}

\def\ho{\operatorname{Ho}}
\newcommand{\iso}[1]{\operatorname{iso}(#1)}

\def\ext{\operatorname{Ext}}

\def\st{\stackrel} 

\def\To{\longrightarrow}

\newcommand{\ad}{\operatorname{add}}

\renewcommand{\ker}{\operatorname{Ker}}
\newcommand{\im}{\operatorname{Im}}

\numberwithin{equation}{section}

\newdir{ >}{{}*!/-7pt/@{>}}
\newdir{> }{{}*!/10pt/@{>}}
\newdir{>> }{{}*!/10pt/@{>>}}

\begin{document}

\title{A note on ${K}$-theory and triangulated derivators}%
\author{Fernando Muro}%
\address{Universidad de Sevilla,
Facultad de Matem\'aticas,
Departamento de \'Algebra,
Avda. Reina Mercedes s/n,
41012 Sevilla, Spain}
\email{fmuro@us.es}
\author{George Raptis}%
\address{Universit\"{a}t Osnabr\"{u}ck, 
Institut f\"{u}r Mathematik, 
Albrechtstrasse 28a,
49069 Osnabr\"{u}ck, Germany}
\email{graptis@mathematik.uni-osnabrueck.de}

\thanks{The first  author was partially supported
by the Spanish Ministry of 
Science and Innovation under the grants  MTM2007-63277 and MTM2010-15831, by the Government of Catalonia under the grant SGR-119-2009,  and by the Andalusian Ministry of Economy, Innovation and Science under the grant FQM-5713.}

\begin{abstract}
In this paper we show an example of two differential graded algebras that have the same derivator $K$-theory but non-isomorphic Waldhausen $K$-theory. We also prove that Maltsiniotis's comparison and localization conjectures for derivator $K$-theory cannot be simultaneously true.
\end{abstract}

\maketitle


\section{Introduction}

\emph{Derivators} were introduced by Grothendieck \cite{derivateurs} as an intermediate step between a (suitable) category $\C{W}$ with weak equivalences and its homotopy category $\ho(\C{W})$, obtained by formally inverting the weak equivalences. In addition to the homotopy category $\ho(\C{W})$, the derivator $\D{D}(\C{W})$ contains all the homotopy categories of diagrams $\ho(\C{W}^{I^{\op}})$ indexed by the opposite of a finite direct  category~$I$ together with the data of the various derived functors associated to the change of indexing category. 
Formally, a derivator $\D{D}$ is a $2$-functor
$$\D{D}\colon \mathbf{Dir}_{f}^{\op}\To\mathbf{Cat}$$
satisfying a set of axioms inspired by the example  $\D{D}(\C{W})(I)=\ho(\C{W}^{I^{\op}})$. Here $\mathbf{Cat}$ is the $2$-category of small categories and $\mathbf{Dir}_{f}\subset \mathbf{Cat}$ is the $2$-full sub-$2$-category of finite direct   categories, i.e. those categories whose nerve has a finite number of non-degenerate simplices.

Maltsiniotis, after Franke \cite{frankepre}, added an extra axiom to define \emph{triangulated derivators} \cite{ktdt}. An important  example that was worked out by Keller \cite{dtce} is the derivator of the category $\C{W}=C^{b}(\C{E})$ of bounded complexes in an exact category $\C{E}$, where the weak equivalences are the quasi-isomorphisms in the abelian envelope of $\C{E}$. This derivator will be denoted $\der(\C{E})$ instead of $\der(C^{b}(\C{E}))$.  Maltsiniotis defined the $K$-theory $K(\D{D})$ of a triangulated derivator $\D{D}$, and stated three conjectures: \emph{comparison}, \emph{localization} and \emph{additivity}. The proof of the additivity conjecture by Cisinski and Neeman \cite{adkt} showed that Maltsiniotis's $K$-theory produces a $K$-theoretically sound construction. The combination of the other two conjectures claims that derivator $K$-theory, as a functor from a suitably defined category of small triangulated derivators to spaces, recovers Quillen's $K$-theory of exact categories and satisfies localization. Schlichting \cite{kttc} showed that there cannot exist a functor from small triangulated categories to spaces that has these two properties.

The comparison conjecture says that a certain natural map from Quillen's to Maltsiniotis's $K$-theory space, 
\[
\rho \colon K(\C{E})\To K(\D{D}(\C{E})),
\]
is a weak homotopy equivalence for every exact category $\C{E}$, i.e., the induced homomorphisms~$\rho_{n}=\pi_{n}(\rho)$, $$\rho_n \colon K_n(\C{E})\To K_n(\D{D}(\C{E})),$$
are isomorphisms for every $n \geq 0$. Maltsiniotis proved this for $n=0$ \cite{ktdt}. The conjecture is also true for $n=1$ \cite{malt}. Moreover, the following more general version holds. If $\C{W}$ is a Waldhausen category with cylinders and a saturated class of weak equivalences, then $\D{D}(\C{W})$ is a  \emph{right pointed derivator} \cite{ciscd}. Garkusha \cite{sdckt1} considered the derivator $K$-theory of this wider class of derivators. There is a natural comparison map 
\begin{equation}\label{mun}
\mu \colon K(\C{W})\To K(\D{D}(\C{W})),
\end{equation}
from Waldhausen's to Garkusha's $K$-theory, and so natural comparison homomorphisms $\mu_{n}=\pi_{n}(\mu)$, $n\geq0$,
\begin{equation*}
\mu_{n}\colon K_{n}(\C{W})\To K_{n}(\D{D}(\C{W})),
\end{equation*}
such that Maltsiniotis's comparison map 
is obtained by precomposing with the Gillet--Waldhausen equivalence \cite{rrthakt,tt,tckt},
$$\rho\colon K(\C{E})\st{\sim}\To K(C^{b}(\C{E}))\st{\mu}\To K(\D{D}(\C{E})).$$
The homomorphisms $\mu_{n}$ are isomorphisms for $n=0,1$ \cite{malt}. 

To\"en and Vezzosi showed in \cite{ktsc} that $\mu$ cannot be an equivalence for all $\C{W}$ in all higher dimensions. More precisely, taking $\C{W}=\mathcal{R}_{f}(X)$ to be the category of finite $CW$-complexes $Y$ relative to a space $X$ equipped with a retraction $Y\r X$ \cite[2.1]{akts}, they noticed that the natural comparison map  $K(\C{W})\r K(\D{D}(\C{W}))$ cannot be a weak homotopy equivalence in general, since the action of the topological monoid of self-equivalences of $X$ on $K(\C{W})$ does not factor through an aspherical space. Despite the fact that the derivator encodes 
a lot of the homotopy theory of~$\C{W}$, it still forgets the higher homotopical information of the 
self-equivalences of~$\C{W}$, and in some sense that is exactly everything that it forgets \cite{pcthqd}.

Thus the failure of the generalized comparison conjecture in this case is due to the homotopical complexity of the algebraic $K$-theory of spaces. This could suggest that 
it may however be true if we restrict to categories of algebraic origin such as, for instance,   the category $\C{W}=\perf{A}$ of perfect right modules over a differential graded algebra $A$. Note that $\D{D}(\perf{A})$ is a triangulated derivator while $\D{D}(\mathcal{R}_{f}(X))$ is not. 

Our first theorem shows that the generalized comparison conjecture is false even in this algebraic context. Moreover, it shows that $K_{n}(\perf{A})$ need not be isomorphic to $K_{n}(\D{D}(\perf{A}))$ even abstractly as abelian groups. This question was so far open.

Let $A=\mathbb{F}_p[x^{\pm1}]$ be the graded algebra of Laurent polynomials with a degree $-1$ indeterminate $x$, $|x|=-1$, over the field with $p$ elements, $p$ a prime number, endowed with the trivial differential $d=0$.

Let $B$ be the  graded algebra over the integers $\mathbb{Z}$ generated by $e$ and  $x^{\pm 1}$ in degree $|e|=|x|=-1$ with the following additional relations,
\begin{align*}
e^2&=0,&
ex+xe&=x^2,
\end{align*}
endowed with the degree $1$ differential $d$ defined by
\begin{align*}
d(e)&=p,&
d(x)&=0.
\end{align*}
Notice that $H^{*}(B)=A$.

\begin{thm}\label{counterexample}
The differential graded algebras $A$ and $B$ have equivalent derivator $K$-theory, $$K(\D{D}(\perf{A})) \simeq K(\D{D}(\perf{B})),$$ but different Waldhausen $K$-theory  $K_i(\perf{A})\ncong K_i(\perf{B})$ in dimension $i=4$ if $p>3$, in dimension $i=5$ if $p=3$, and in dimension $i=3$ if $p=2$.
\end{thm}

This is something close, although not actually  a counterexample to Maltsiniotis's original comparison conjecture. Nevertheless, we succeed in showing that the comparison conjecture cannot be true if the localization conjecture holds. The localization conjecture says that any pair of composable pseudo-natural transformations between triangulated derivators $$\D{D}'\To\D{D}\To\D{D}''$$ that yields a short exact sequence of triangulated categories $\D{D}'(I)\r\D{D}(I)\r\D{D}''(I)$ at any finite direct  category $I$, induces a homotopy fibration in $K$-theory spaces, $$K(\D{D}')\To K(\D{D})\To K(\D{D}'').$$

\begin{thm} \label{disagreement}
Derivator $K$-theory does not satisfy both the comparison and localization conjectures. 
\end{thm}

\section{The comparison homomorphisms}

Let $\C{M}$ be a stable pointed model category \cite{hmc}. An object $X$ in the triangulated category $\C{T}=\ho(\C{M})$ is \emph{compact} if the functor $\C{T}(X,-)$ preserves coproducts. In most cases, the Waldhausen category $\C{W}$ will essentially be the full subcategory of compact cofibrant objects in some $\C{M}$. For any direct category $I$ we equip the diagram category $\C{M}^{I}$ with the model structure where weak equivalences and fibrations are defined levelwise \cite[Theorem 5.1.3]{hmc}.

The poset $\deltan{n}=\{0<1<\cdots<n\}$ is the free category generated by the following Dynkin quiver of type $A_{n+1}$,
\begin{equation*}
Q_n=\quad
0\r 1\r\cdots\r n.
\end{equation*}
A compact cofibrant object in $\C{M}^{\deltan{n}}$ is a sequence of cofibrations between compact cofibrant objects in $\C{M}$,
$$X_{0}\rightarrowtail X_{1}\rightarrowtail\cdots \rightarrowtail X_{n}.$$
Objects in the category $S_{n+1}\C{W}$ are such sequences together with choices of cofibers $X_{ij} = X_{j}/X_{i}$, $0\leq i\leq j\leq n$, such that $X_{ii}=0$ is the distinguished zero object of~$\C{W}$. The morphisms are defined as commutative diagrams in the obvious way. It is clear that $S_{n+1}\C{W}$ is equivalent to the full subcategory of $\C{M}^{\deltan{n}}$ that consists of the compact cofibrant objects, the equivalence being the functor forgetting the choice of cofibers. The categories $\ho(\C{W}^{\deltan{n}})$ and $\ho(S_{n+1} \C{W})$ are also equivalent. For~$n=0$,~$S_{0}\C{W}$ is the category with one object $0$ and one morphism (the identity). These categories assemble to a simplicial Waldhausen category $S_{\bullet}\C{W}$ where $d_{i+1}$ and $s_{i+1}$ are induced by the usual functors $d^{i}\colon \deltan{n-1}\r\deltan{n}$ and $s^{i}\colon \deltan{n+1}\r\deltan{n}$, $0\leq i\leq n$, the degeneracy $s_{0}$ adds $0$ at the beginning of the sequence,
$$0\rightarrowtail X_{0}\rightarrowtail X_{1}\rightarrowtail\cdots \rightarrowtail X_{n},$$
and the face $d_{0}$ uses the choice of cofibers
$$ X_{1}/X_{0}\rightarrowtail\cdots \rightarrowtail X_{n}/X_{0}.$$

Waldhausen's $K$-theory space of $\C{W}$ is defined to be the loop space of the geometric realization of the simplicial subcategory of weak equivalences \cite{akts},
$$K(\C{W})= \Omega|wS_{\bullet}\C{W}|.$$
In most cases of interest, the Waldhausen category $\C{W}$ produces a right pointed derivator $\D{D}(\C{W})$ \cite{ciscd}. The $K$-theory space of $\D{D}(\C{W})$ is naturally weakly equivalent to 
$$K(\D{D}(\C{W}))\simeq \Omega|\iso{\ho (S_{\bullet}\C{W})}|,$$
where $\iso{-}$ denotes the subcategory of isomorphisms, compare \cite{sdckt2,malt}. The natural comparison map \eqref{mun} is induced by the canonical simplicial functor $S_{\bullet}\C{W}\r \ho (S_{\bullet}\C{W})$, which takes weak equivalences to isomorphisms, and therefore gives rise to a map $$\mu\colon \Omega|wS_{\bullet}\C{W}|\To \Omega|\iso{\ho (S_{\bullet}\C{W})}|.$$

\section{The proof of Theorem \ref{counterexample}}

For any commutative Frobenius ring $R$ we denote $\sMod{R}$ the \emph{stable category of $R$-modules}, i.e., the quotient of the category $\Mod{R}$ of $R$-modules by the ideal of morphisms which factor through a projective object. We denote $\C{M}_{R}$ the stable pointed model category structure on $\Mod{R}$ where fibrations are the epimorphisms, cofibrations are the monomorphisms, and weak equivalences are those homomorphisms which project to an isomorphism in $\sMod{R}$ \cite[Theorem 2.2.12]{hmc}. The homotopy category is $\ho(\C{M}_{R})=\sMod{R}$ and the compact objects are essentially the finitely generated $R$-modules. We denote $\C{W}_{R}$ the Waldhausen subcategory of finitely generated $R$-modules, whose underlying category is denoted by $\mod{R}$. Its homotopy category $\ho(\C{W}_{R})=\smod{R}$ is the full subcategory of $\sMod{R}$ spanned by the finitely generated $R$-modules.

In this section we will compute $\ho(\C{W}_{R}^{\deltan{n}})$ for $R$ a commutative noetherian local ring with principal maximal ideal $(\alpha)\neq 0$ such that $\alpha^{2}=0$. We denote $k=R/\alpha$ the residue field. The ring $R$ is Frobenius by \cite[Theorems 15.27 and 16.14]{lmr}, since the socle of $R$ is $(\alpha)\cong k$. The upshot of the computation will be that the category $\ho(\C{W}_{R}^{\deltan{n}})$ is  determined up to equivalence by the field $k$. 

The first important observation, for $n=0$, is that the full inclusion of the category of finite dimensional $k$-vector spaces $\mod{k}\subset \mod{R}$ induced by the unique ring homomorphism $q\colon R\twoheadrightarrow k$ gives rise to an equivalence of categories
$$\mod{k}\st{\sim}\To \smod{R}.$$
For any $n\geq0$, this functor induces a section of the canonical functor $\ho(\C{W}_{R}^{\deltan{n}})\r\ho(\C{W}_{R})^{\deltan{n}}$, since the following composite is an equivalence,
\begin{equation}\label{splitting}
\mod{k}^{\deltan{n}}\subset \mod{R}^{\deltan{n}}=\C{W}_{R}^{\deltan{n}}\To\ho(\C{W}_{R}^{\deltan{n}})\To\ho(\C{W}_{R})^{\deltan{n}}=\smod{R}^{\deltan{n}}.
\end{equation}

\begin{prop}
The additive functor $\ho(\C{W}_{R}^{\deltan{n}})\r\ho(\C{W}_{R})^{\deltan{n}}$ is full, essentially surjective, and reflects isomorphisms. In particular it induces a bijection between the corresponding monoids of isomorphism classes of objects.
\end{prop}

\begin{proof}
By \cite[Proposition 2.15]{ciscd} the functor is full and essentially surjective. 
It reflects isomorphisms since weak equivalences in $\C{W}_{R}$ satisfy the strong saturation property.
\end{proof}

The category $\mod{k}^{\deltan{n}}$ is well known  since it is the category of finite-dimensional representations of the quiver $Q_{n}$ over the field $k$. The monoid of isomorphism classes of objects is freely generated by the following set of indecomposable representations,
\begin{align*}
E_{n}&=\{M_{i,j}\; ;\; 0\leq i\leq j\leq n\},\\
M_{i,j}&=\underbrace{0\r\cdots\r 0}_{i}\r \underbrace{k\st{1}\r\cdots\st{1}\r k}_{j-i+1}\r \underbrace{0\r\cdots\r 0}_{n-j}.
\end{align*}

Given a category $\C{C}$ and a set of objects $E$ in $\C{C}$ we denote $\C{C}_{E}\subset\C{C}$ the full subcategory  spanned by the objects in $E$. We now describe $\mod{k}^{\deltan{n}}_{E_{n}}$.
Consider the poset
$$P_{n}=\{(i,j)\; ;\;  0\leq i\leq j\leq n\}\subset \mathbb{Z}\times \mathbb{Z}$$
and the corresponding $k$-linear category $k[P_{n}]$. We refer to \cite{rso} for the basics on the ring and module theory of linear categories. Then 
there is an isomorphism of $k$-linear categories
$$k[P_{n}]^{\op}/((i-1,i-1)\leq (i,i))_{i=1}^{n}\st{\cong}\To\mod{k}^{\deltan{n}}_{E_{n}}$$ 
sending  $(i,j)$ to $M_{i,j}$ and $(i_{1},j_{1})\leq (i_{2},j_{2})$ to the morphism $M_{i_{2},j_{2}}\r M_{i_{1},j_{1}}$ given by
$$\begin{array}{ccccccccccccccccccccccccccccc}
0\hspace{-7pt}&\r\hspace{-7pt}&\cdots\hspace{-7pt}&\r\hspace{-7pt}&0\hspace{-7pt}&\r\hspace{-7pt}&0\hspace{-7pt}&\r\hspace{-7pt}&\cdots\hspace{-7pt}&\r\hspace{-7pt}&0\hspace{-7pt}&\r\hspace{-7pt}&k\hspace{-7pt}&\st{1}\r\hspace{-7pt}&\cdots\hspace{-7pt}&\st{1}\r\hspace{-7pt}&k\hspace{-7pt}&
\st{1}\r\hspace{-7pt}&k\hspace{-7pt}&\st{1}\r\hspace{-7pt}&\cdots\hspace{-7pt}&\st{1}\r\hspace{-7pt}&k\hspace{-7pt}&\r\hspace{-7pt}&0\hspace{-7pt}&\r\hspace{-7pt}&\cdots\hspace{-7pt}&\r\hspace{-7pt}&0\\
\downarrow\hspace{-7pt}&\hspace{-7pt}&\hspace{-7pt}&\hspace{-7pt}& \downarrow\hspace{-7pt}&\hspace{-7pt}& \downarrow\hspace{-7pt}&\hspace{-7pt}&\hspace{-7pt}&\hspace{-7pt}& \downarrow\hspace{-7pt}&\hspace{-7pt}& \hspace{-6pt}\downarrow{\scriptstyle 1}\hspace{-20pt}&\hspace{-7pt}&\hspace{-7pt}&\hspace{-7pt}& \hspace{-6pt}\downarrow{\scriptstyle 1}\hspace{-20pt}& \hspace{-7pt}& \downarrow\hspace{-7pt}&\hspace{-7pt}&\hspace{-7pt}&\hspace{-7pt}& \downarrow\hspace{-7pt}&\hspace{-7pt}& \downarrow\hspace{-7pt}&\hspace{-7pt}&\hspace{-7pt}&\hspace{-7pt}& \downarrow\vspace{-2pt}\\
0\hspace{-7pt}&\r\hspace{-7pt}&\cdots\hspace{-7pt}&\r\hspace{-7pt}&0\hspace{-7pt}&\r\hspace{-7pt}&k\hspace{-7pt}&\st{1}\r\hspace{-7pt}&\cdots\hspace{-7pt}&\st{1}\r\hspace{-7pt}&k\hspace{-7pt}&\st{1}\r\hspace{-7pt}&k\hspace{-7pt}&\st{1}\r\hspace{-7pt}&\cdots\hspace{-7pt}&\st{1}\r\hspace{-7pt}&k\hspace{-7pt}&
\r\hspace{-7pt}&0\hspace{-7pt}&\r\hspace{-7pt}&\cdots\hspace{-7pt}&\r\hspace{-7pt}&0\hspace{-7pt}&\r\hspace{-7pt}&0\hspace{-7pt}&\r\hspace{-7pt}&\cdots\hspace{-7pt}&\r\hspace{-7pt}&0
\end{array}$$
Equivalently, $\mod{k}^{\deltan{n}}_{E_{n}}$ is isomorphic to the $k$-linear category of the following quiver~$\Gamma_{Q_{n}}$
\begin{center}
\includegraphics[scale=.18]{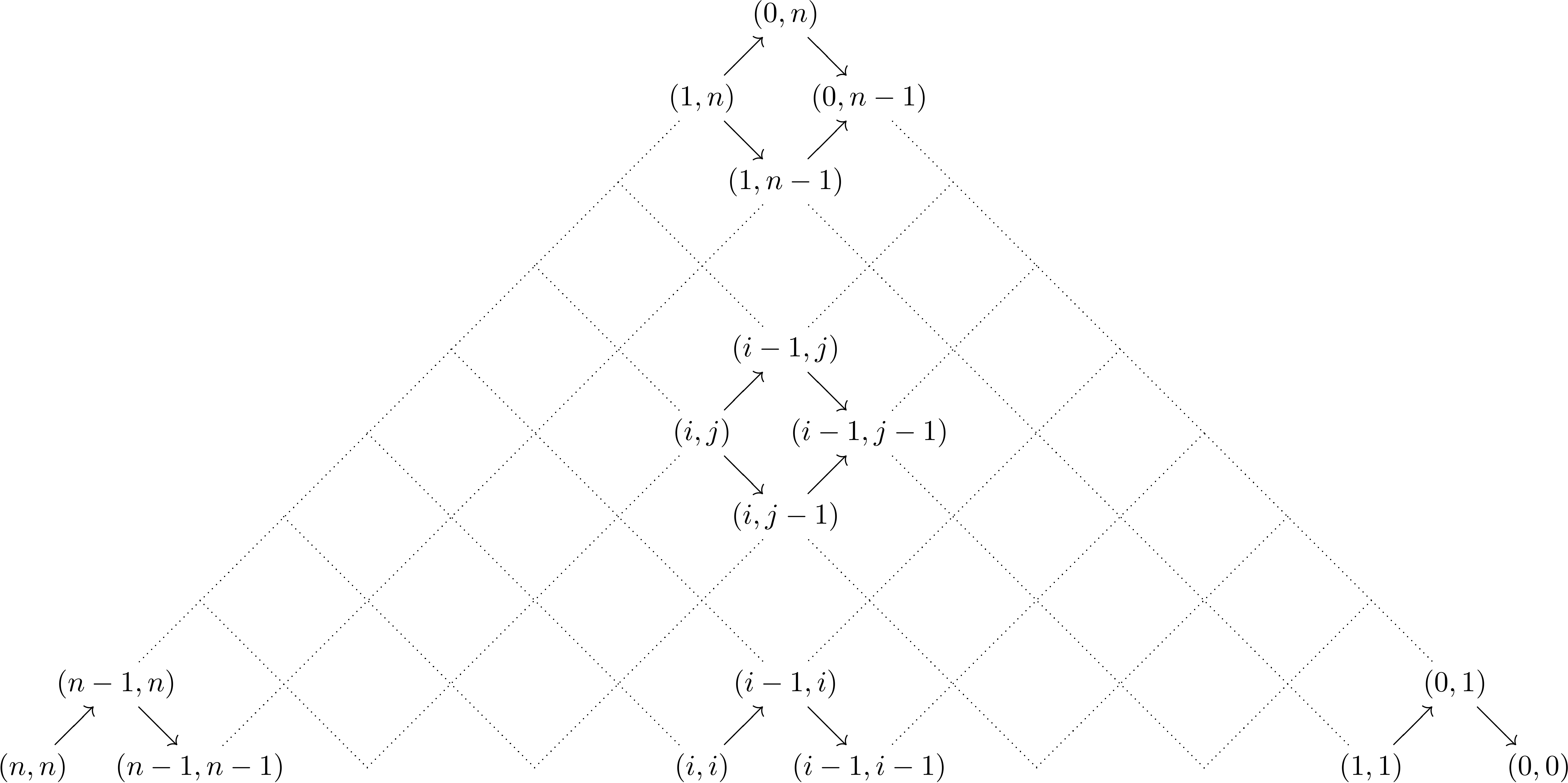}
\end{center}
e.g. for $n=3$ 
$$\Gamma_{Q_{3}}\quad=\quad\begin{array}{c}
\xymatrix@!=5pt{
&&&(0,3)\ar[rd]&\\
&&(1,3)\ar[ru]\ar[rd]&&(0,2)\ar[rd]\\
&(2,3)\ar[rd]\ar[ru]&&(1,2)\ar[ru]\ar[rd]&&(0,1)\ar[rd]\\
(3,3)\ar[ru]&&(2,2)\ar[ru]&&(1,1)\ar[ru]&&(0,0)\\
}\end{array}$$
with relations:
\begin{itemize}
\item The square 
$\begin{array}{c}\xymatrix@!=5pt{
&(i-1,j)\ar[rd]&\\
(i,j)\ar[ru]\ar[rd]&&(i-1,j-1)\\
&(i,j-1)\ar[ru]&}\end{array}$ is commutative for $0< i< j\leq n$,

\item The composite $\begin{array}{c}\xymatrix@!=5pt{
&(i-1,i)\ar[rd]&\\
(i,i)\ar[ru]&&(i-1,i-1)\\
&&}\end{array}$ is trivial for  $0< i\leq n$.
\end{itemize}
This category will be denoted by $k_{\tau}\Gamma_{Q_n}$ since it is (isomorphic to) the \emph{mesh category} of the quiver $\Gamma_{Q_n}$ with translation $\tau(i,j)=(i+1,j+1)$, $0\leq i\leq j< n$, which is the \emph{Auslander-Reiten quiver} of $Q_{n}$, see \cite{csrt}. The arrows $\nearrow$ and $\searrow$ in $\Gamma_{Q_{n}}$ correspond to the following morphisms,
\begin{align}
\label{ne}
\nearrow\quad=&\quad\begin{array}{ccccccccccccccccccccccccccccc}
0\hspace{-7pt}&\r\hspace{-7pt}&\cdots\hspace{-7pt}&\r\hspace{-7pt}&0\hspace{-7pt}&\r\hspace{-7pt}&0\hspace{-7pt}&\r\hspace{-7pt}&k\hspace{-7pt}&\st{1}\r\hspace{-7pt}&\cdots\hspace{-7pt}&\st{1}\r\hspace{-7pt}&k\hspace{-7pt}&\r\hspace{-7pt}&0\hspace{-7pt}&\r\hspace{-7pt}&\cdots\hspace{-7pt}&\r\hspace{-7pt}&0\\
\downarrow\hspace{-7pt}&\hspace{-7pt}&\hspace{-7pt}&\hspace{-7pt}& \downarrow\hspace{-7pt}&\hspace{-7pt}& \downarrow\hspace{-7pt}&\hspace{-7pt}&  \hspace{-6pt}\downarrow{\scriptstyle 1}\hspace{-20pt}&\hspace{-7pt}&\hspace{-7pt}&\hspace{-7pt}& \hspace{-6pt}\downarrow{\scriptstyle 1}\hspace{-20pt}& \hspace{-7pt}& \downarrow\hspace{-7pt}&\hspace{-7pt}&\hspace{-7pt}&\hspace{-7pt}& \downarrow\vspace{-2pt}\\
0\hspace{-7pt}&\r\hspace{-7pt}&\cdots\hspace{-7pt}&\r\hspace{-7pt}&0\hspace{-7pt}&\r\hspace{-7pt}&k\hspace{-7pt}&\st{1}\r\hspace{-7pt}&k\hspace{-7pt}&\st{1}\r\hspace{-7pt}&\cdots\hspace{-7pt}&\st{1}\r\hspace{-7pt}&k\hspace{-7pt}&
\r\hspace{-7pt}&0\hspace{-7pt}&\r\hspace{-7pt}&\cdots\hspace{-7pt}&\r\hspace{-7pt}&0
\end{array}\\
\label{se}
\searrow \quad=&\quad \begin{array}{ccccccccccccccccccccccccccccc}
0\hspace{-7pt}&\r\hspace{-7pt}&\cdots\hspace{-7pt}&\r\hspace{-7pt}&0\hspace{-7pt}&\r\hspace{-7pt}&k\hspace{-7pt}&\st{1}\r\hspace{-7pt}&\cdots\hspace{-7pt}&\st{1}\r\hspace{-7pt}&k\hspace{-7pt}&\st{1}\r\hspace{-7pt}&k\hspace{-7pt}&\r\hspace{-7pt}&0\hspace{-7pt}&\r\hspace{-7pt}&\cdots\hspace{-7pt}&\r\hspace{-7pt}&0\\
\downarrow\hspace{-7pt}&\hspace{-7pt}&\hspace{-7pt}&\hspace{-7pt}& \downarrow\hspace{-7pt}&\hspace{-7pt}&\hspace{-6pt}\downarrow{\scriptstyle 1}\hspace{-20pt}&\hspace{-7pt}&\hspace{-7pt}&\hspace{-7pt}&  \hspace{-6pt}\downarrow{\scriptstyle 1}\hspace{-20pt}&\hspace{-7pt}& \downarrow\hspace{-7pt}&\hspace{-7pt}&\downarrow\hspace{-7pt}&\hspace{-7pt}& \hspace{-7pt}&\hspace{-7pt}&\downarrow\vspace{-2pt}\\
0\hspace{-7pt}&\r\hspace{-7pt}&\cdots\hspace{-7pt}&\r\hspace{-7pt}&0\hspace{-7pt}&\r\hspace{-7pt}&k\hspace{-7pt}&\st{1}\r\hspace{-7pt}&\cdots\hspace{-7pt}&\st{1}\r\hspace{-7pt}&k\hspace{-7pt}&\r\hspace{-7pt}&0\hspace{-7pt}&
\r\hspace{-7pt}&0\hspace{-7pt}&\r\hspace{-7pt}&\cdots\hspace{-7pt}&\r\hspace{-7pt}&0
\end{array}
\end{align}

Recall that given a $k$-linear category $\C{C}$ and a $\C{C}$-bimodule $D$ the \emph{semidirect product} $D\rtimes\C{C}$ is the category with the same objects as $\C{C}$, morphism $k$-modules
$$(D\rtimes\C{C})(X,Y)=D(X,Y)\oplus\C{C}(X,Y),$$
and composition defined by
$$(b,g)(a,f)=(b\cdot f+g\cdot a,gf).$$

Let $D_{n}$ be the $k_{\tau}\Gamma_{Q_n}$-bimodule with generators
$$\varphi_{i}\in D_{n}((0,i-1),(i,n)),\quad 0< i\leq n,$$
and relations
\begin{align*}
((i+1,n)\r(i,n))\cdot\varphi_{i+1}&=\varphi_{i}\cdot((0,i)\r(0,i-1)),\quad 0< i<n;\\
((1,n)\r(0,n))\cdot \varphi_{1}&=0.
\end{align*}
This $k_{\tau}\Gamma_{Q_n}$-bimodule is
$$D_{n}((i_{1},j_{1}),(i_{2},j_{2}))\cong k,\quad i_{1}+1\leq i_{2}\leq j_{1}+1\leq j_{2}, $$
generated by $$((i_{2},n)\r(i_{2}, j_{2}))\cdot\varphi_{i_{2}}\cdot((i_{1},j_{1})\r(0,i_{2}-1)),$$
and zero otherwise. It is a well-known fact from representation theory that, under the isomorphism $k_{\tau}\Gamma_{Q_n}\cong\mod{k}^{\deltan{n}}_{E_{n}}$, we have a bimodule isomorphism $D_{n}\cong\ext^{1}_{\mod{k}^{\deltan{n}}}$ sending $\varphi_{i}$ to the element represented by the following extension
$$
\begin{array}{ccccccccccccccccc}
0\hspace{-7pt}&\r\hspace{-7pt}&\cdots\hspace{-7pt}&\r\hspace{-7pt}& 0\hspace{-7pt}&\r\hspace{-7pt}& k\hspace{-7pt}&\st{1}\r\hspace{-7pt}&\cdots\hspace{-7pt}&\st{1}\r\hspace{-7pt}& k\\
\downarrow\hspace{-7pt}&\hspace{-7pt}&\hspace{-7pt}&\hspace{-7pt}&\downarrow\hspace{-7pt}&\hspace{-7pt}& \hspace{-6pt}\downarrow{\scriptstyle 1}\hspace{-20pt}&\hspace{-7pt}&\hspace{-7pt}&\hspace{-7pt}&\hspace{-6pt}\downarrow{\scriptstyle 1}\hspace{-13pt}\vspace{-3pt}\\
k\hspace{-7pt}&\st{1}\r\hspace{-7pt}&\cdots\hspace{-7pt}&\st{1}\r\hspace{-7pt}& k\hspace{-7pt}&\st{1}\r\hspace{-7pt}& k\hspace{-7pt}&\st{1}\r\hspace{-7pt}&\cdots\hspace{-7pt}&\st{1}\r\hspace{-7pt}& k\\
\hspace{-6pt}\downarrow{\scriptstyle 1}\hspace{-20pt}&\hspace{-7pt}&\hspace{-7pt}&\hspace{-7pt}&\hspace{-6pt}\downarrow{\scriptstyle 1}\hspace{-20pt}&\hspace{-7pt}& \downarrow\hspace{-7pt}&\hspace{-7pt}&\hspace{-7pt}&\hspace{-7pt}&\downarrow\vspace{-3pt}\\
k\hspace{-7pt}&\st{1}\r\hspace{-7pt}&\cdots\hspace{-7pt}&\st{1}\r\hspace{-7pt}& k\hspace{-7pt}&\r\hspace{-7pt}& 0\hspace{-7pt}&\r\hspace{-7pt}&\cdots\hspace{-7pt}&\r\hspace{-7pt}& 0
\end{array}$$
compare \cite{rtaa}, hence
$$D_{n}\rtimes k_{\tau}\Gamma_{Q_n} \cong \ext^{1}_{\mod{k}^{\deltan{n}}}\rtimes \mod{k}^{\deltan{n}}_{E_{n}}.$$

For $0\leq i\leq j\leq n$ we consider the cofibrant replacement $r_{i,j}\colon \tilde{M}_{i,j}\st{\sim}\twoheadrightarrow M_{i,j}$ of $M_{i,j}$ in $\C{W}_{R}^{\deltan{n}}$ given by
$$
\begin{array}{ccccccccccccccccc}
0\hspace{-7pt}&\r\hspace{-7pt}&\cdots\hspace{-7pt}&\r\hspace{-7pt}& 0\hspace{-7pt}&\r\hspace{-7pt}& k\hspace{-7pt}&\st{1}\r\hspace{-7pt}&\cdots\hspace{-7pt}&\st{1}\r\hspace{-7pt}& k\hspace{-7pt}&\st{\alpha}\r\hspace{-7pt}& R\hspace{-7pt}&\st{1}\r\hspace{-7pt}&\cdots\hspace{-7pt}&\st{1}\r\hspace{-7pt}& R\\
\downarrow\hspace{-7pt}&\hspace{-7pt}&\hspace{-7pt}&\hspace{-7pt}&\downarrow\hspace{-7pt}&\hspace{-7pt}& \hspace{-6pt}\downarrow{\scriptstyle 1}\hspace{-20pt}&\hspace{-7pt}&\hspace{-7pt}&\hspace{-7pt}& \hspace{-6pt}\downarrow{\scriptstyle 1}\hspace{-20pt}&\hspace{-7pt}&\downarrow\hspace{-7pt}&\hspace{-7pt}&\hspace{-7pt}&\hspace{-7pt}&\downarrow\vspace{-2pt}\\
0\hspace{-7pt}&\r\hspace{-7pt}&\cdots\hspace{-7pt}&\r\hspace{-7pt}& 0\hspace{-7pt}&\r\hspace{-7pt}& k\hspace{-7pt}&\st{1}\r\hspace{-7pt}&\cdots\hspace{-7pt}&\st{1}\r\hspace{-7pt}& k\hspace{-7pt}&\r\hspace{-7pt}& 0\hspace{-7pt}&\r\hspace{-7pt}&\cdots\hspace{-7pt}&\r\hspace{-7pt}& 0
\end{array}$$
Notice that $r_{i,n}$ is the identity, $0\leq i\leq n$.

\begin{prop}\label{iso1}
There is an isomorphism of categories
$$D_{n}\rtimes k_{\tau}\Gamma_{Q_n}\st{\cong}\To \ho(\C{W}_{R}^{\deltan{n}})_{E_{n}}$$
which is defined on the second factor by the isomorphism $k_{\tau}\Gamma_{Q_n}\cong  \mod{k}^{\deltan{n}}_{E_{n}}$ and the splitting \eqref{splitting}, and on the first factor it sends ${\varphi}_{i}$, $0<i\leq n$, to the inverse of $r_{0,i-1}$ in the homotopy category composed with the morphism $\tilde{\varphi}_{i}\colon \tilde{M}_{0,i-1}\r{M}_{i,n}$ below,
$$
\begin{array}{ccccccccccc}
k\hspace{-7pt}&\st{1}\r\hspace{-7pt}&\cdots\hspace{-7pt}&\st{1}\r\hspace{-7pt}& k\hspace{-7pt}&\st{\alpha}\r\hspace{-7pt}& R\hspace{-7pt}&\st{1}\r\hspace{-7pt}&\cdots\hspace{-7pt}&\st{1}\r\hspace{-7pt}& R\\
\downarrow\hspace{-7pt}&\hspace{-7pt}&\hspace{-7pt}&\hspace{-7pt}&\downarrow\hspace{-7pt}&\hspace{-7pt}&\hspace{-6pt}\downarrow{\scriptstyle q}\hspace{-20pt}&\hspace{-7pt}&\hspace{-7pt}&\hspace{-7pt}&\hspace{-6pt}\downarrow{\scriptstyle q}\hspace{-12pt}\vspace{-2pt}\\
0\hspace{-7pt}&\r\hspace{-7pt}&\cdots\hspace{-7pt}&\r\hspace{-7pt}& 0\hspace{-7pt}&\r\hspace{-7pt}& k\hspace{-7pt}&\st{1}\r\hspace{-7pt}&\cdots\hspace{-7pt}&\st{1}\r\hspace{-7pt}& k
\end{array}$$

\end{prop}

\begin{proof}
The arrows $M_{i,j}\r M_{i-1,j}$ and $M_{i,j}\r M_{i,j-1}$ in \eqref{ne} and \eqref{se} lift to the following two morphisms $\tilde{M}_{i,j}\r \tilde{M}_{i-1,j}$ and $\tilde{M}_{i,j}\r \tilde{M}_{i,j-1}$ between cofibrant replacements, respectively,
\begin{align}
\label{ne2}
\begin{array}{ccccccccccccccccccccccccccccc}
0\hspace{-7pt}&\r\hspace{-7pt}&\cdots\hspace{-7pt}&\r\hspace{-7pt}&0\hspace{-7pt}&\r\hspace{-7pt}&0\hspace{-7pt}&\r\hspace{-7pt}&k\hspace{-7pt}&\st{1}\r\hspace{-7pt}&\cdots\hspace{-7pt}&\st{1}\r\hspace{-7pt}&k\hspace{-7pt}&\st{\alpha}\r\hspace{-7pt}&R\hspace{-7pt}&\st{1}\r\hspace{-7pt}&\cdots\hspace{-7pt}&\st{1}\r\hspace{-7pt}&R\\
\downarrow\hspace{-7pt}&\hspace{-7pt}&\hspace{-7pt}&\hspace{-7pt}& \downarrow\hspace{-7pt}&\hspace{-7pt}& \downarrow\hspace{-7pt}&\hspace{-7pt}&  \hspace{-6pt}\downarrow{\scriptstyle 1}\hspace{-20pt}&\hspace{-7pt}&\hspace{-7pt}&\hspace{-7pt}& \hspace{-6pt}\downarrow{\scriptstyle 1}\hspace{-20pt}& \hspace{-7pt}&\hspace{-6pt}\downarrow{\scriptstyle 1}\hspace{-20pt}&\hspace{-7pt}&\hspace{-7pt}&\hspace{-7pt}& \hspace{-6pt}\downarrow{\scriptstyle 1}\hspace{-13pt}\vspace{-2pt}\\
0\hspace{-7pt}&\r\hspace{-7pt}&\cdots\hspace{-7pt}&\r\hspace{-7pt}&0\hspace{-7pt}&\r\hspace{-7pt}&k\hspace{-7pt}&\st{1}\r\hspace{-7pt}&k\hspace{-7pt}&\st{1}\r\hspace{-7pt}&\cdots\hspace{-7pt}&\st{1}\r\hspace{-7pt}&k\hspace{-7pt}&
\st{\alpha}\r\hspace{-7pt}&R\hspace{-7pt}&\st{1}\r\hspace{-7pt}&\cdots\hspace{-7pt}&\st{1}\r\hspace{-7pt}&R
\end{array}\\
\label{se2}
 \begin{array}{ccccccccccccccccccccccccccccc}
0\hspace{-7pt}&\r\hspace{-7pt}&\cdots\hspace{-7pt}&\r\hspace{-7pt}&0\hspace{-7pt}&\r\hspace{-7pt}&k\hspace{-7pt}&\st{1}\r\hspace{-7pt}&\cdots\hspace{-7pt}&\st{1}\r\hspace{-7pt}&k\hspace{-7pt}&\st{1}\r\hspace{-7pt}&k\hspace{-7pt}&\st{\alpha}\r\hspace{-7pt}&R\hspace{-7pt}&\st{1}\r\hspace{-7pt}&\cdots\hspace{-7pt}&\st{1}\r\hspace{-7pt}&R\\
\downarrow\hspace{-7pt}&\hspace{-7pt}&\hspace{-7pt}&\hspace{-7pt}& \downarrow\hspace{-7pt}&\hspace{-7pt}&\hspace{-6pt}\downarrow{\scriptstyle 1}\hspace{-20pt}&\hspace{-7pt}&\hspace{-7pt}&\hspace{-7pt}&  \hspace{-6pt}\downarrow{\scriptstyle 1}\hspace{-20pt}&\hspace{-7pt}&  \hspace{-6pt}\downarrow{\scriptstyle \alpha}\hspace{-20pt}&\hspace{-7pt}& \hspace{-6pt}\downarrow{\scriptstyle 1}\hspace{-20pt}&\hspace{-7pt}& \hspace{-7pt}&\hspace{-7pt}& \hspace{-6pt}\downarrow{\scriptstyle 1}\hspace{-13pt}\vspace{-2pt}\\
0\hspace{-7pt}&\r\hspace{-7pt}&\cdots\hspace{-7pt}&\r\hspace{-7pt}&0\hspace{-7pt}&\r\hspace{-7pt}&k\hspace{-7pt}&\st{1}\r\hspace{-7pt}&\cdots\hspace{-7pt}&\st{1}\r\hspace{-7pt}&k\hspace{-7pt}&\st{\alpha}\r\hspace{-7pt}&R\hspace{-7pt}&
\st{1}\r\hspace{-7pt}&R\hspace{-7pt}&\st{1}\r\hspace{-7pt}&\cdots\hspace{-7pt}&\st{1}\r\hspace{-7pt}&R
\end{array}
\end{align}
Given $(i_{1},j_{1})\geq (i_{2},j_{2})$ we denote
\begin{align*}
\nu_{i_{1},j_{1};i_{2},j_{2}}&\colon M_{i_{1},j_{1}}\To M_{i_{2},j_{2}},&
\tilde{\nu}_{i_{1},j_{1};i_{2},j_{2}}&\colon \tilde{M}_{i_{1},j_{1}}\To \tilde{M}_{i_{2},j_{2}},
\end{align*}
the  unique composite of \eqref{ne}'s and \eqref{se}'s, and of 
\eqref{ne2}'s and~\eqref{se2}'s, respectively. We have $r_{i_{2},j_{2}}\tilde{\nu}_{i_{1},j_{1};i_{2},j_{2}}={\nu}_{i_{1},j_{1};i_{2},j_{2}}r_{i_{1},j_{1}}$.

One can easily check by inspection that the  epimorphism $r_{i_{1},j_{1}}$ induces the monomorphism
\begin{align*}
\mod{R}^{\deltan{n}}(M_{i_{1},j_{1}},M_{i_{2},j_{2}})&\hookrightarrow
\mod{R}^{\deltan{n}}(\tilde{M}_{i_{1},j_{1}},M_{i_{2},j_{2}})
\end{align*}
given by:
\begin{itemize}
 \item ${k}\hookrightarrow {k}$, an isomorphism, if $i_{2}\leq i_{1}\leq j_{2}\leq j_{1}$;
 \item $0\hookrightarrow {k}$ if $i_2\leq j_{1}+1\leq j_{2}$, with target generated by
\begin{equation}\label{nontrivial}
\nu_{i_{2},n;i_{2},j_{2}}
\tilde{\varphi}_{i_{2}} \tilde{\nu}_{i_{1},j_{1};0,i_{2}-1},
\end{equation}
if $i_{1}+1\leq i_{2}$; 
and   if $i_{1}\geq i_{2}$ by 
\begin{equation}\label{htrivial}
\begin{array}{ccccccccccccccccccccccccccccc}
0\hspace{-7pt}&\r\hspace{-7pt}&\cdots\hspace{-7pt}&\r\hspace{-7pt}&0\hspace{-7pt}&\r\hspace{-7pt}&0\hspace{-7pt}&\r\hspace{-7pt}&\cdots\hspace{-7pt}&\r\hspace{-7pt}&0\hspace{-7pt}&\r\hspace{-7pt}&k\hspace{-7pt}&\st{1}\r\hspace{-7pt}&\cdots\hspace{-7pt}&\st{1}\r\hspace{-7pt}&k\hspace{-7pt}&
\st{\alpha}\r\hspace{-7pt}&R\hspace{-7pt}&\st{1}\r\hspace{-7pt}&\cdots\hspace{-7pt}&\st{1}\r\hspace{-7pt}&R\hspace{-7pt}&\st{1}\r\hspace{-7pt}&R\hspace{-7pt}&\st{1}\r\hspace{-7pt}&\cdots\hspace{-7pt}&\st{1}\r\hspace{-7pt}&R\\
\downarrow\hspace{-7pt}&\hspace{-7pt}&\hspace{-7pt}&\hspace{-7pt}& \downarrow\hspace{-7pt}&\hspace{-7pt}& \downarrow\hspace{-7pt}&\hspace{-7pt}&\hspace{-7pt}&\hspace{-7pt}& \downarrow\hspace{-7pt}&\hspace{-7pt}& \hspace{-6pt}\downarrow{\scriptstyle 0}\hspace{-20pt}&\hspace{-7pt}&\hspace{-7pt}&\hspace{-7pt}& \hspace{-6pt}\downarrow{\scriptstyle 0}\hspace{-20pt}& \hspace{-7pt}& \hspace{-6pt}\downarrow{\scriptstyle q}\hspace{-20pt}&\hspace{-7pt}&\hspace{-7pt}&\hspace{-7pt}& \hspace{-6pt}\downarrow{\scriptstyle q}\hspace{-20pt}&\hspace{-7pt}& \downarrow\hspace{-7pt}&\hspace{-7pt}&\hspace{-7pt}&\hspace{-7pt}& \downarrow\vspace{-2pt}\\
0\hspace{-7pt}&\r\hspace{-7pt}&\cdots\hspace{-7pt}&\r\hspace{-7pt}&0\hspace{-7pt}&\r\hspace{-7pt}&k\hspace{-7pt}&\st{1}\r\hspace{-7pt}&\cdots\hspace{-7pt}&\st{1}\r\hspace{-7pt}&k\hspace{-7pt}&\st{1}\r\hspace{-7pt}&k\hspace{-7pt}&\st{1}\r\hspace{-7pt}&\cdots\hspace{-7pt}&\st{1}\r\hspace{-7pt}&k\hspace{-7pt}&
\st{1}\r\hspace{-7pt}&k\hspace{-7pt}&\st{1}\r\hspace{-7pt}&\cdots\hspace{-7pt}&\st{1}\r\hspace{-7pt}&k\hspace{-7pt}&\r\hspace{-7pt}&0\hspace{-7pt}&\r\hspace{-7pt}&\cdots\hspace{-7pt}&\r\hspace{-7pt}&0
\end{array}
\end{equation}
\item $0\hookrightarrow 0$ otherwise.
 \end{itemize}
The morphism \eqref{htrivial} is trivial in $\ho(\C{W}_{R}^{\deltan{n}})$ since it factors through a contractible object as follows,
\begin{equation*}
\begin{array}{ccccccccccccccccccccccccccccc}
0\hspace{-7pt}&\r\hspace{-7pt}&\cdots\hspace{-7pt}&\r\hspace{-7pt}&0\hspace{-7pt}&\r\hspace{-7pt}&0\hspace{-7pt}&\r\hspace{-7pt}&\cdots\hspace{-7pt}&\r\hspace{-7pt}&0\hspace{-7pt}&\r\hspace{-7pt}&k\hspace{-7pt}&\st{1}\r\hspace{-7pt}&\cdots\hspace{-7pt}&\st{1}\r\hspace{-7pt}&k\hspace{-7pt}&
\st{\alpha}\r\hspace{-7pt}&R\hspace{-7pt}&\st{1}\r\hspace{-7pt}&\cdots\hspace{-7pt}&\st{1}\r\hspace{-7pt}&R\hspace{-7pt}&\st{1}\r\hspace{-7pt}&R\hspace{-7pt}&\st{1}\r\hspace{-7pt}&\cdots\hspace{-7pt}&\st{1}\r\hspace{-7pt}&R\\
\downarrow\hspace{-7pt}&\hspace{-7pt}&\hspace{-7pt}&\hspace{-7pt}& \downarrow\hspace{-7pt}&\hspace{-7pt}& \downarrow\hspace{-7pt}&\hspace{-7pt}&\hspace{-7pt}&\hspace{-7pt}& \downarrow\hspace{-7pt}&\hspace{-7pt}& \hspace{-6pt}\downarrow{\scriptstyle \alpha}\hspace{-20pt}&\hspace{-7pt}&\hspace{-7pt}&\hspace{-7pt}& \hspace{-6pt}\downarrow{\scriptstyle \alpha}\hspace{-20pt}& \hspace{-7pt}& \hspace{-6pt}\downarrow{\scriptstyle 1}\hspace{-20pt}&\hspace{-7pt}&\hspace{-7pt}&\hspace{-7pt}& \hspace{-6pt}\downarrow{\scriptstyle 1}\hspace{-20pt}&\hspace{-7pt}& \hspace{-6pt}\downarrow{\scriptstyle 1}\hspace{-20pt}&\hspace{-7pt}&\hspace{-7pt}&\hspace{-7pt}& \hspace{-6pt}\downarrow{\scriptstyle 1}\hspace{-13pt}\vspace{-2pt}\\
0\hspace{-7pt}&\r\hspace{-7pt}&\cdots\hspace{-7pt}&\r\hspace{-7pt}&0\hspace{-7pt}&\r\hspace{-7pt}&0\hspace{-7pt}&\r\hspace{-7pt}&\cdots\hspace{-7pt}&\r\hspace{-7pt}&0\hspace{-7pt}&\r\hspace{-7pt}&R\hspace{-7pt}&\st{1}\r\hspace{-7pt}&\cdots\hspace{-7pt}&\st{1}\r\hspace{-7pt}&R\hspace{-7pt}&
\st{1}\r\hspace{-7pt}&R\hspace{-7pt}&\st{1}\r\hspace{-7pt}&\cdots\hspace{-7pt}&\st{1}\r\hspace{-7pt}&R\hspace{-7pt}&\st{1}\r\hspace{-7pt}&R\hspace{-7pt}&\st{1}\r\hspace{-7pt}&\cdots\hspace{-7pt}&\st{1}\r\hspace{-7pt}&R\\
\downarrow\hspace{-7pt}&\hspace{-7pt}&\hspace{-7pt}&\hspace{-7pt}& \downarrow\hspace{-7pt}&\hspace{-7pt}& \downarrow\hspace{-7pt}&\hspace{-7pt}&\hspace{-7pt}&\hspace{-7pt}& \downarrow\hspace{-7pt}&\hspace{-7pt}& \hspace{-6pt}\downarrow{\scriptstyle q}\hspace{-20pt}&\hspace{-7pt}&\hspace{-7pt}&\hspace{-7pt}& \hspace{-6pt}\downarrow{\scriptstyle q}\hspace{-20pt}& \hspace{-7pt}& \hspace{-6pt}\downarrow{\scriptstyle q}\hspace{-20pt}&\hspace{-7pt}&\hspace{-7pt}&\hspace{-7pt}& \hspace{-6pt}\downarrow{\scriptstyle q}\hspace{-20pt}&\hspace{-7pt}& \downarrow\hspace{-7pt}&\hspace{-7pt}&\hspace{-7pt}&\hspace{-7pt}& \downarrow\vspace{-2pt}\\
0\hspace{-7pt}&\r\hspace{-7pt}&\cdots\hspace{-7pt}&\r\hspace{-7pt}&0\hspace{-7pt}&\r\hspace{-7pt}&k\hspace{-7pt}&\st{1}\r\hspace{-7pt}&\cdots\hspace{-7pt}&\st{1}\r\hspace{-7pt}&k\hspace{-7pt}&\st{1}\r\hspace{-7pt}&k\hspace{-7pt}&\st{1}\r\hspace{-7pt}&\cdots\hspace{-7pt}&\st{1}\r\hspace{-7pt}&k\hspace{-7pt}&
\st{1}\r\hspace{-7pt}&k\hspace{-7pt}&\st{1}\r\hspace{-7pt}&\cdots\hspace{-7pt}&\st{1}\r\hspace{-7pt}&k\hspace{-7pt}&\r\hspace{-7pt}&0\hspace{-7pt}&\r\hspace{-7pt}&\cdots\hspace{-7pt}&\r\hspace{-7pt}&0
\end{array}
\end{equation*}
On the contrary  \eqref{nontrivial} cannot be trivial in $\ho(\C{W}_{R}^{\deltan{n}})$ since it contains a commutative square of the form
$$\begin{array}{ccc}
k\hspace{-7pt}&\st{\alpha}\r\hspace{-7pt}&R\\
\downarrow\hspace{-7pt}&\hspace{-7pt}& \hspace{-6pt}\downarrow{\scriptstyle q}\hspace{-13pt}\\
0\hspace{-7pt}&\r\hspace{-7pt}&k
\end{array}$$
which induces the non-trivial morphism $1\colon k\r k$ in $\ho(\C{W}_{R})\simeq\mod{k}$ on cofibers (i.e. cokernels) of horizontal arrows (which are cofibrations).

These computations prove that the kernel of $\ho(\C{W}_{R}^{\deltan{n}})_{E_{n}}\r\ho(\C{W}_{R})^{\deltan{n}}_{E_{n}}$  is the ideal $I_{n}$ 
with $$I_{n}(M_{i_{1},j_{1}},M_{i_{2},j_{2}})\cong k,\quad i_{1}+1\leq i_{2}\leq j_{1}+1\leq j_{2}, $$
generated by  
the inverse of $r_{i_{1},j_{1}}$ composed with
\eqref{nontrivial}; and zero otherwise.

In order to prove that the functor in the statement is well defined we only have to check that $I_{n}$ is a square-zero ideal. The only non-trivial cases to check are the composites of  generators $$M_{i_{1},j_{1}}\To M_{i_{2},j_{2}}\To M_{i_{3},j_{3}},\quad i_{1}+3\leq i_{2}+2\leq i_{3}+1\leq j_{1}+2\leq j_{2}+1\leq j_{3}.$$
Such a composite, precomposed with the isomorphism $r_{i_{1},j_{1}}$ in the homotopy category, is
\begin{align*}
&\hspace{-40pt}\nu_{i_{3},n;i_{3},j_{3}}
\tilde{\varphi}_{i_{3}} \tilde{\nu}_{i_{2},j_{2};0,i_{3}-1}
\tilde{\nu}_{i_{2},n;i_{2},j_{2}}
\tilde{\varphi}_{i_{2}} \tilde{\nu}_{i_{1},j_{1};0,i_{2}-1}\\
={}&
\nu_{i_{3},n;i_{3},j_{3}}
\tilde{\varphi}_{i_{3}}
\tilde{\nu}_{i_{2},n;0,i_{3}-1}
\tilde{\varphi}_{i_{2}} \tilde{\nu}_{i_{1},j_{1};0,i_{2}-1}.
\end{align*}
The morphism $\tilde{\nu}_{i_{2},n;0,i_{3}-1}$ is
$$\begin{array}{ccccccccccccccccccccccccccccc}
0\hspace{-7pt}&\r\hspace{-7pt}&\cdots\hspace{-7pt}&\r\hspace{-7pt}&0\hspace{-7pt}&\r\hspace{-7pt}&k\hspace{-7pt}&\st{1}\r\hspace{-7pt}&\cdots\hspace{-7pt}&\st{1}\r\hspace{-7pt}&k\hspace{-7pt}&
\st{1}\r\hspace{-7pt}&k\hspace{-7pt}&\st{1}\r\hspace{-7pt}&\cdots\hspace{-7pt}&\st{1}\r\hspace{-7pt}&k\\
\downarrow\hspace{-7pt}&\hspace{-7pt}&\hspace{-7pt}&\hspace{-7pt}& \downarrow\hspace{-7pt}&\hspace{-7pt}& \hspace{-6pt}\downarrow{\scriptstyle 1}\hspace{-20pt}&\hspace{-7pt}&\hspace{-7pt}&\hspace{-7pt}& \hspace{-6pt}\downarrow{\scriptstyle 1}\hspace{-20pt}& \hspace{-7pt}& \hspace{-6pt}\downarrow{\scriptstyle \alpha}\hspace{-20pt}&\hspace{-7pt}&\hspace{-7pt}&\hspace{-7pt}& \hspace{-6pt}\downarrow{\scriptstyle \alpha}\hspace{-13pt}\vspace{-2pt}\\
k\hspace{-7pt}&\st{1}\r\hspace{-7pt}&\cdots\hspace{-7pt}&\st{1}\r\hspace{-7pt}&k\hspace{-7pt}&\st{1}\r\hspace{-7pt}&k\hspace{-7pt}&\st{1}\r\hspace{-7pt}&\cdots\hspace{-7pt}&\st{1}\r\hspace{-7pt}&k\hspace{-7pt}&
\st{\alpha}\r\hspace{-7pt}&R\hspace{-7pt}&\st{1}\r\hspace{-7pt}&\cdots\hspace{-7pt}&\st{1}\r\hspace{-7pt}&R
\end{array}$$
therefore $\tilde{\varphi}_{i_{3}}
\tilde{\nu}_{i_{2},n;0,i_{3}-1}=0$, and hence also the composite $M_{i_{1},j_{1}}\r M_{i_{3},j_{3}}$.

Now that we have proved that the functor is well defined, the previous computation of $I_{n}$ automatically shows that the functor is actually an isomorphism, see the descriptions of the bimodule $D_{n}$. Moreover, $\ho(\C{W}_{R}^{\deltan{n}})_{E_{n}}\r\ho(\C{W}_{R})^{\deltan{n}}_{E_{n}}$ and its splitting correspond to the projection and the inclusion of the second factor of $D_{n}\rtimes k_{\tau}\Gamma_{Q_n}$, respectively.
\end{proof}

For any $k$-linear category $\C{C}$ we denote $\C{C}_{+}$ the $k$-linear category obtained by formally adding  a zero object $0$. This is the left adjoint of the forgetful functor from $k$-linear categories with zero object to $k$-linear categories.

We define the simplicial $k$-linear category with zero object $\mathcal{D}_{\bullet}(k)$ as $\mathcal{D}_{0}(k)=0$, $\mathcal{D}_{n+1}(k)=(D_{n}\rtimes k_{\tau}\Gamma_{Q_n})_{+}$, $n\geq 0$, faces and degeneracies defined on objects by
\begin{align*}
d_{t+1}(i,j)&=\left\{
\begin{array}{ll}
(i-1,j-1),&0\leq t <i\leq j\leq n;\\
(i,j-1),&0\leq i\leq t \leq j\leq n,\; i<j;\\
0,&0\leq i= t = j\leq n;\\
(i,j),&0\leq i\leq j<t\leq n;\\
\end{array}
\right.\\
d_{0}(i,j)&=
\left\{
\begin{array}{ll}
(j+1,n-1),&0= i\leq j< n;\\
0,&0= i\leq  j= n;\\
(i-1,j-1),&0< i\leq j\leq n;\\
\end{array}
\right.\\
s_{t+1}(i,j)&=\left\{
\begin{array}{ll}
(i+1,j+1),&-1\leq t <i\leq j\leq n;\\
(i,j+1),&0\leq i\leq t \leq j\leq n;\\
(i,j),&0\leq i\leq j<t\leq n;\\
\end{array}
\right.
\end{align*}
on morphisms in $k_{\tau}\Gamma_{Q_n}$ as
\begin{align*}
d_{t}((i_{1},j_{1})\geq (i_{2},j_{2}))&=\left\{
\begin{array}{ll}
0,&\text{if }d_{t}(i_{1},j_{1})=0 \text{ or } d_{t}(i_{2},j_{2})=0;\\
d_{t}(i_{1},j_{1})\geq d_{t}(i_{2},j_{2}),&\text{otherwise};
\end{array}
\right.\\
s_{t}((i_{1},j_{1})\geq (i_{2},j_{2}))&=s_{t}(i_{1},j_{1})\geq s_{t} (i_{2},j_{2});
\end{align*}
and on generators of $D_{n}$ as
\begin{align*}
d_{t+1}(\varphi_{i})&=\left\{
\begin{array}{ll}
0,&t=0,\;i=1;\\
\varphi_{i-1},&0\leq t<i\leq n,\; i>1;\\
\varphi_{i},&0< i\leq t\leq n;
\end{array}
\right.\\
d_{0}(\varphi_{i})&=1_{(i-1,n-1)},\quad 0<i\leq n;\\
s_{t+1}(\varphi_{i})&=\left\{
\begin{array}{ll}
\varphi_{i+1},&0\leq t <i\leq n;\\
\varphi_{i},&0< i\leq t\leq n;
\end{array}
\right.\\
s_{0}(\varphi_{i})&=\varphi_{i+1}\cdot ((1,i)\r(0,i)), \quad 0< i\leq  n.
\end{align*}
We leave the reader the tedious but straightforward task to check that $\mathcal{D}_{\bullet}(k)$ is well defined. Notice that this simplicial category only depends on the field $k$. Actually, it is functorial in $k$. Moreover, the inclusion of the second factors $$\mathcal{E}_{n}(k):=(k_{\tau}\Gamma_{Q_n})_{+} \subset \mathcal{D}_{n}(k)$$ define a simplicial $k$-linear subcategory with zero object $\mathcal{E}_{\bullet}(k)\subset \mathcal{D}_{\bullet}(k)$, however the face~$d_{0}$ does not take all elements in $D_{n}$ to elements in $D_{n-1}$. 

Consider the following set of objects of $S_{n+1}\C{W}_{R}$,
$$\tilde{E}_{n}=\{\tilde{M}_{i,j}\; ;\; 0\leq i\leq j\leq n\},\quad n\geq0;\qquad \tilde{E}_{-1}=\emptyset;$$
with the choices of cofibers indicated in the following short exact sequences,
\begin{align}
\label{choices} 0\hookrightarrow0\twoheadrightarrow0,&&
0\hookrightarrow k\st{1}\twoheadrightarrow k,&&
k\st{1}\hookrightarrow k\twoheadrightarrow 0,\\
\nonumber 0\hookrightarrow R\st{1}\twoheadrightarrow R,&&
k\st{\alpha}\hookrightarrow R\st{q}\twoheadrightarrow k,&&
R\st{1}\hookrightarrow R\twoheadrightarrow 0.
\end{align}
Note that the full inclusions $\ho(S_{n}\C{W}_{R})_{\tilde{E}_{n-1}\cup\{0\}}\subset \ho(S_{n}\C{W}_{R})$ define a  simplicial subcategory with zero object $\ho(S_{\bullet}\C{W}_{R})_{\tilde{E}_{\bullet-1}\cup\{0\}}\subset \ho(S_{\bullet}\C{W}_{R})$.

\begin{lem}\label{iso2}
There is a unique isomorphism $$\xi(R,\alpha)\colon \mathcal{D}_{\bullet}(k)\cong \ho(S_{\bullet}\C{W}_{R})_{\tilde{E}_{\bullet-1}}$$ such that the cofibrant replacements $r_{i,j}\colon \tilde{M}_{i,j}\twoheadrightarrow M_{i,j}$ define natural isomorphisms indicated with double arrow in the following diagram, $0\leq i\leq j\leq n$,
$$\xymatrix@C=40pt{\mathcal{D}_{n+1}(k)\ar[r]^-{\xi(R,\alpha)_{n+1}}_-{\cong}_<(.5){\begin{array}{c}\vspace{5pt}\end{array}}="a"\ar[d]_-{\eqref{iso1}}^-{\cong}
&\ho(S_{n+1}\C{W}_{R})_{\tilde{E}_{n}\cup\{0\}}\ar[d]^-{\text{full inclusion}}\\
\ho(\C{W}_{R}^{\deltan{n}})_{E_{n}\cup\{0\}}\ar[r]_-{\text{full inclusion}}^<(.1){
\begin{array}{c}\vspace{5pt}\end{array}}="b"&\ho(\C{W}_{R}^{\deltan{n}})
\ar@{=>}"a";"b"^-{\scriptscriptstyle\cong}}$$
\end{lem}

This lemma is just an observation. Actually,  this observation  motivated our definition of $\mathcal{D}_{\bullet}(k)$.

The \emph{additivization} of a $k$-linear category $\C{C}$ is the category $\C{C}^{\ad}$ whose objects are sequences $(C_1,\dots,C_n)$ 
of finite length $n\geq 0$ of  objects  in $\C{C}$, and whose morphisms are matrices,
$$\C{C}^{\ad}((C_i)_{i=1}^n,(C_j')_{j=1}^m)=\!\!\!\!\!\bigoplus_{
\begin{array}{c}
\vspace{-15pt}
\\
\scriptstyle 1\leq i\leq n
\vspace{-3pt}
\\
\scriptstyle 1\leq j\leq m
\end{array}
}\!\!\!\!\!\C{C}(C_i,C_j').$$
This is an additive $k$-linear category with  zero object $()$ the empty sequence (of legth $n=0$).
The additivization is the  left adjoint (in the weak $2$-categorical sense) of the forgetful $2$-functor from small additive $k$-linear categories to small $k$-linear linear categories (weak because the counit is given by the choice of direct sums $C_{1}\oplus\cdots\oplus C_{n}$ and zero object $0$ in $\C{C}$). Moreover, if $E$ is a set of generators of the monoid of isomorphism classes of objects of $\C{C}$ then an adjoint
$\C{C}^{\ad}_E\st{\sim}\r\C{C}$ of the full inclusion
$\C{C}_E\subset\C{C}$ is an equivalence.

\begin{cor} \label{iso3}
The isomorphism $\xi(R,\alpha)$ in Lemma \ref{iso2}  induces a simplicial $k$-linear functor $$\tilde{\xi}(R,\alpha)\colon \mathcal{D}_{\bullet}(k)^{\ad}\st{\sim}\To \ho(S_{\bullet}\C{W}_{R})$$ which is a levelwise equivalence. In particular, restricting to isomorphisms we obtain a simplicial functor which is a levelwise equivalence,
$$\iso{\tilde{\xi}(R,\alpha)}\colon \iso{\mathcal{D}_{\bullet}(k)^{\ad}}\st{\sim}\To \iso{\ho(S_{\bullet}\C{W}_{R})}.$$
\end{cor}

Now we are ready to prove Theorem \ref{counterexample}.

\begin{proof}[Proof of Theorem \ref{counterexample}]
Let $p$ be a prime. 
The rings $\mathbb{F}_{p}[\varepsilon]/\varepsilon^{2}$ and $\mathbb{Z}/p^{2}$ are commutative noetherian local with non-trivial maximal ideal generated by $\alpha=\varepsilon$ and $p$, respectively, and such that $\alpha^{2}=0$, therefore they satisfy the assumptions of this section.

Dugger and Shipley proved that $\C{M}_{\mathbb{F}_{p}[\varepsilon]/\varepsilon^{2}}$ and $\C{M}_{\mathbb{Z}/p^{2}}$ are Quillen equivalent to the categories $\Mod{A}$ and $\Mod{B}$ of right modules over the differential graded algebras defined in the introduction, respectively, where weak equivalences are quasi-isomorphisms and fibrations are levelwise surjections   \cite[Corollary 4.4]{curious} (we here use cohomological degree, unlike Dugger and Shipley). This, together with the approximation theorem in $K$-theory (e.g. see \cite[Th\'eor\`eme 2.15]{ikted}, \cite[Corollary 3.10]{ktde}) shows that we can replace everywhere $\perf{A}$ and $\perf{B}$ with  $\C{W}_{\mathbb{F}_{p}[\varepsilon]/\varepsilon^{2}}$ and $\C{W}_{\mathbb{Z}/p^{2}}$, respectively.

Now the equivalence in the statement  of the theorem is obtained by taking the loop space of the geometric realization of the  second map in the previous corollary. The last claim for $p>3$ is a computation made in \cite[1.6]{kttc}, see also \cite[Remark 4.9]{curious}. 

For $p=2$ we can apply the computations $K_{3}(\mathbb{F}_{2})\cong\mathbb{Z}/3$, $K_{2}(\mathbb{F}_{2})=0$ \cite[Theorem 8]{otcktglgoff}, $K_{2}(\mathbb{F}_{2}[\varepsilon]/\varepsilon^{2})=0$ \cite{k2nd}, and $K_{2}(\mathbb{Z}/4)\cong\mathbb{Z}/2$ \cite[Corollary 4.4]{k2dvr} to the long exact sequences arising from the homotopy fibration in \cite[1.5]{kttc} in order to obtain exact sequences,
$$\begin{array}{cc}
\mathbb{Z}/3\r K_{3}(\C{W}_{\mathbb{F}_{2}[\varepsilon]/\varepsilon^{2}})\r 0,\qquad&
\mathbb{Z}/3\r K_{3}(\C{W}_{\mathbb{Z}/4})\r \mathbb{Z}/2\r 0.
\end{array}$$
Therefore $K_{3}(\C{W}_{\mathbb{Z}/4})$ contains a cyclic group of order $2$ unlike $K_{3}(\C{W}_{\mathbb{F}_{2}[\varepsilon]/\varepsilon^{2}})$, so they cannot be isomorphic.

For $p=3$ we  apply  the calculations $K_{4}(\mathbb{F}_{3}[\varepsilon]/\varepsilon^{2})=0$  and $K_{4}(\mathbb{Z}/9)\cong\mathbb{Z}/3$ in \cite{otaktozpn} together with the well known $K_{5}(\mathbb{F}_{3})\cong\mathbb{Z}/26$ and $K_{4}(\mathbb{F}_{3})=0$ from \cite[Theorem 8]{otcktglgoff}. In this case we obtain the following exact sequences,
$$\begin{array}{cc}
\mathbb{Z}/26\r K_{5}(\C{W}_{\mathbb{F}_{3}[\varepsilon]/\varepsilon^{2}})\r 0,\qquad&
\mathbb{Z}/26\r K_{5}(\C{W}_{\mathbb{Z}/9})\r \mathbb{Z}/3\r 0.
\end{array}$$
Hence $K_{5}(\C{W}_{\mathbb{Z}/9})$ contains a cyclic group of order $3$ and $K_{5}(\C{W}_{\mathbb{F}_{2}[\varepsilon]/\varepsilon^{2}})$ does not, so they are not isomorphic.

\end{proof}


\begin{rem}
Note that the approximation theorem in $K$-theory implies that the stable model categories $\C{M}_{\mathbb{F}_{p}[\varepsilon]/\varepsilon^{2}}$ and $\C{M}_{\mathbb{Z}/p^{2}}$ cannot 
be connected by a zigzag of Quillen equivalences. By a result of Renaudin \cite[Th\'eor\`eme 3.3.2]{pcthqd}, it follows that the associated triangulated derivators are not equivalent. However, by Corollary \ref{iso3}, their difference cannot be detected on the part that is involved in the definition of derivator $K$-theory.
\end{rem}

\section{The proof of Theorem \ref{disagreement}}

Given a commutative noetherian ring $R$, let $\mod{R}$ be the abelian (in particular exact) category of finitely generated  $R$-modules and $j\colon \proj{R} \to \mod{R}$ the inclusion of the full exact subcategory of projectives. There is an induced morphism between the associated triangulated derivators,
\begin{equation*}
 \der(j) \colon \D{D}(\proj{R}) \To \D{D}(\mod{R}),
\end{equation*}
which is given for every $I$ in $\mathbf{Dir}_{f}$ by the inclusion of a thick subcategory,
\begin{equation*}
 \der(j)(I) \colon \db(\proj{R}^{I^{\op}}) \To \db(\mod{R}^{I^{\op}}).
\end{equation*}
The pointwise Verdier quotients define a prederivator,
\begin{align*}
\underline{\der}(R) \colon \mathbf{Dir}_{f}^{\op} &\To \mathbf{Cat}, \\ 
I&\; \mapsto \db(\mod{R}^{I^{\op}})/ \db(\proj{R}^{I^{\op}}).
\end{align*}
Let $e$ denote the terminal category and, given $A$ in  $\mathbf{Dir}_{f}$ and an object $a$ in $A$, let  $i_{A,a}\colon e \to A$ be the inclusion of $a$. 

\begin{prop} \label{conservative}
For every finite direct  category $A$, the collection of functors $i^*_{A,a}\colon \underline{\der}(R)(A) \to \underline{\der}(R)(e)$, $a$ in $A$, is conservative, i.e., a morphism $f$ in $\underline{\der}(R)(A)$ is an isomorphism if and only if $i^*_{A,a}(f)$ is an isomorphism for every $a$ in $A$.
\end{prop}

In the proof of this proposition we will use the following lemma.

\begin{lem}
Let $P^{*}$ be a bounded above complex of projective $R$-modules  which is quasi-isomorphic to a bounded   complex of projective $R$-modules $Q^{*}$ with $Q^{m}=0$ for $m\leq N$. Then $\im[d^{m}\colon P^{m-1}\r P^{m}]=\ker[d^{m+1}\colon P^{m}\r P^{m+1}]$ is a projective $R$-module for all $m<N$.
\end{lem}

\begin{proof}
Both $P^{*}$ and $Q^{*}$ are fibrant and cofibrant in the model category $C(\Mod{R})$ of complexes of $R$-modules \cite[Lemma 2.3.6 and Theorem 2.3.11]{hmc}, therefore there exists a quasi-isomorphism $f\colon Q^{*}\st{\sim}\r P^{*}$. The mapping cone $C^{*}_{f}$ of $f$ is a bounded above acyclic complex of projective $R$-modules, therefore $C^{*}_{f}$ is contractible \cite[Lemma 2.3.8]{hmc}, in particular 
$$\im[d^{m}\colon C_{f}^{m-1}\r C_{f}^{m}]=\ker[d^{m+1}\colon C_{f}^{m}\r C_{f}^{m+1}]$$ is a projective $R$-module for all $m\in \mathbb{Z}$. The complex $C_{f}^{*}$ coincides with $P^{*}$ in degrees $m<N$, hence we are done.
\end{proof}

We can now prove the proposition.

\begin{proof}[Proof of Proposition \ref{conservative}]
Let $f\colon X \to Y$ be a morphism in $\underline{\der}(R)(A)$ such that $i^*_{A,a}(f)$ is an isomorphism for all objects $a$ in $A$. Let $$\tilde{f}\colon \tilde{X} \To \tilde{Y}$$ be a morphism in $\der(\mod{R})(A)$ whose image under the localization functor $$S\colon \der(\mod{R})(A) \To \underline{\der}(R)(A)$$ is isomorphic to~$f$. Since $S$ is a triangulated functor, the distinguished triangle
\begin{displaymath}
 \tilde{X} \st{\tilde{f}}{\To} \tilde{Y} \To \tilde{Z} \To \tilde{X}[1]
\end{displaymath}
in $\der(\mod{R}^{A^{\op}})$ maps to a distinguished triangle isomorphic to 
\begin{displaymath}
X \st{f}{\To} Y \To Z \To X[1]
\end{displaymath}
and so, by the assumption, there is an object $P_a$ in $\db(\proj{R})$ and an isomorphism $g_a \colon P_a \cong i^*_{A,a}(\tilde{Z})$ for every object $a$ in $A$. Let $N$ be an integer such that $P_a^m=0$ for every $m \leq N$ and $a$ in $A$. 

We consider the closed model category, in the sense of Quillen \cite{ha}, of bounded above complexes $C^{-}(\mod{R})$ of finitely generated $R$-modules, where the weak equivalences are the quasi-isomorphisms and the fibrations are the levelwise surjective morphisms. In the derivable category $C^{-}(\mod{R})^{A^{\op}}$ \cite{ciscd}, 
there is a cofibrant resolution given by a pointwise quasi-isomorphism $\tilde{Q} \st{\sim}{\to} \tilde{Z}$ where $\tilde{Q}$ is a diagram of bounded above complexes of finitely generated projective $R$-modules. By the previous lemma, replacing the complex
\begin{displaymath}
\tilde{Q}_a = \; \cdots \r  \tilde{Q}_a^{N-3}\st{d^{N-2}}\To\tilde{Q}_a^{N-2}\st{d^{N-1}}\To \tilde{Q}_a^{N-1} \st{d^{N}}\to \cdots \to \tilde{Q}_a^{n-1} \st{d^n}{\to} \tilde{Q}_a^{n} \to \cdots \to 0 \to \cdots 
\end{displaymath}
with its functorial truncation
\begin{displaymath}
\tau_{\geq N-1} \tilde{Q}_a =\;  \cdots \to 0 \to \ker(d^{N}) \to \tilde{Q}_a^{N-1} \st{d^{N}}\to \cdots \r\tilde{Q}_a^{n-1} \st{d^n}{\to} \tilde{Q}_a^{n}\r\cdots \to 0 \to \cdots 
\end{displaymath}
gives a diagram of bounded complexes of finitely generated projective $R$-modules such that the inclusion $\tau_{\geq N-1} \tilde{Q} \st{\sim}\to \tilde{Q}$ is a pointwise quasi-isomorphism. Hence $\tilde{Z}$ is isomorphic in $\db(\mod{R}^{A^{\op}})$ to an object in $\db(\proj{R}^{A^{\op}})$, so $S(\tilde{f})$ is an isomorphism, as required to show. 
\end{proof}


Let $\tilde{\C{W}}_{R}$ be the Waldhausen category with cylinders with underlying category $C^{b}(\mod{R})$ where  cofibrations are the monomorphisms and  weak equivalences are the morphisms whose mapping cone is quasi-isomorphic to a bounded complex of finitely generated projective $R$-modules. That is precisely the strongly saturated class of morphisms that map into isomorphisms under the functor $\cb(\mod{R}) \to \db(\mod{R})/\db(\proj{R})$.

\begin{prop} \label{auxiliary1}
There is an isomorphism of prederivators 
$\der(\tilde{\C{W}}_{R}) \cong \underline{\der}(R)$. In particular, $\underline{\der}(R)$ is a triangulated derivator.
\end{prop}

\begin{proof}
Let $I$ be a  finite direct category. The universal property of localizations implies that the dotted arrows can be uniquely filled in the following commutative diagram
\begin{displaymath}
 \xymatrix{
\cb(\mod{R}^{I^{\op}}) \ar[r] \ar[d] & \db(\mod{R}^{I^{\op}}) \ar[r] \ar@{.>}[dl] & \db(\mod{R}^{I^{\op}})/\db(\proj{R}^{I^{\op}}). \ar@{.>}[dll] \\
\ho(\tilde{\C{W}}_{R}^{I^{\op}}) & &
}
\end{displaymath}
Similarly in the following diagram, using Proposition \ref{conservative},
\begin{displaymath}
 \xymatrix{
\cb(\mod{R}^{I^{\op}}) \ar[r] \ar[d] & \db(\mod{R}^{I^{\op}})/\db(\proj{R}^{I^{\op}})  \\
\ho(\tilde{\C{W}}_{R}^{I^{\op}}) \ar@{.>}[ur] & 
}
\end{displaymath}
By the universal property, the functors
\begin{equation*} 
\ho(\tilde{\C{W}}_{R}^{I^{\op}}) \rightleftarrows \db(\mod{R}^{I^{\op}})/\db(\proj{R}^{I^{\op}})
\end{equation*}
are mutually inverse isomorphisms, natural in $I$, hence  the result follows.
\end{proof}

\begin{prop} \label{auxiliary2}
Let $R$ be a commutative Frobenius ring. The full inclusion  $G \colon \C{W}_{R}=\mod{R} \to \cb(\mod{R})=\tilde{\C{W}}_{R}$ of complexes concentrated in degree $0$ induces an equivalence of triangulated derivators 
$$\der(G) \colon \der(\C{W}_{R}) \st{\sim}{\To} \der(\tilde{\C{W}}_{R}).$$
\end{prop}

\begin{proof}
By \cite[Theorem 2.1]{dcse}, $G$ induces an equivalence
\begin{equation*}
 \ho(G) \colon \ho(\C{W}_{R})=\smod{R} \st{\sim}{\To} \db(\mod{R})/\db(\proj{R})\cong\ho(\tilde{\C{W}}_{R}).
\end{equation*}
The functor $G$ is a right exact functor between strongly saturated right derivable categories, therefore by \cite[Th\'eor\`em 3.19]{ciscd} it induces an equivalence $\der(G)$ of triangulated derivators.
\end{proof}

We are now ready to prove our second theorem.

\begin{proof}[Proof of Theorem \ref{disagreement}]
As in the previous section, let $R$ be a commutative noetherian local ring with principal maximal ideal $(\alpha)\neq 0$ such that $\alpha^{2}=0$,  $k=R/\alpha$ the residue field, and $q\colon R\twoheadrightarrow k$ the natural projection.  We consider the full simplicial subcategory  $\mathcal{B}_{\bullet}(k)\subset \mathcal{E}_{\bullet}(k)$ spanned by the zero object and $(i,n)$, $0\leq i\leq n$. The following set $E_{n}'$ is a basis of the monoid of isomorphism classes of objects in $S_{n+1}(\mod{k})$,
\begin{align*}
E'_{n}&=\{M_{i,n}\; ;\; 0\leq i\leq n\},\quad n\geq0, \quad E'_{-1}= \emptyset, \\ 
M_{i,n}&=\underbrace{0\r\cdots\r 0}_{i}\r \underbrace{k \st{1}\r\cdots\st{1}\r 
k}_{n-i+1},
\end{align*}
with choices of cofibers as in \eqref{choices}. 
There is defined an isomorphism of simplicial categories preserving the zero object 
$\mathcal{B}_{\bullet}(k)\cong S_{\bullet}(\mod{k})_{E'_{\bullet-1}\cup\{0\}} \colon (i,n)\mapsto M_{i,n}$ sending $(i,n)\r (i-1,n)$ to 
$$\begin{array}{ccccccccccccccccccccccccccccc}
0\hspace{-7pt}&\r\hspace{-7pt}&\cdots\hspace{-7pt}&\r\hspace{-7pt}&0\hspace{-7pt}&\r\hspace{-7pt}&0\hspace{-7pt}&\r\hspace{-7pt}&k\hspace{-7pt}&\st{1}\r\hspace{-7pt}&\cdots\hspace{-7pt}&\st{1}\r\hspace{-7pt}&\hspace{-7pt}k\\
\downarrow\hspace{-7pt}&\hspace{-7pt}&\hspace{-7pt}&\hspace{-7pt}& \downarrow\hspace{-7pt}&\hspace{-7pt}& \downarrow\hspace{-7pt}&\hspace{-7pt}&  \hspace{-6pt}\downarrow{\scriptstyle 1}\hspace{-20pt}&\hspace{-7pt}&\hspace{-7pt}&\hspace{-7pt}& \downarrow{\scriptstyle 1}\vspace{-2pt}\\
0\hspace{-7pt}&\r\hspace{-7pt}&\cdots\hspace{-7pt}&\r\hspace{-7pt}&0\hspace{-7pt}&\r\hspace{-7pt}&k\hspace{-7pt}&\st{1}\r\hspace{-7pt}&k\hspace{-7pt}&\st{1}\r\hspace{-7pt}&\cdots\hspace{-7pt}&\st{1}\r\hspace{-7pt}&\hspace{-7pt}k
\end{array}$$
This isomorphism defines a levelwise equivalence of simplicial categories
$\mathcal{B}_{\bullet}(k)^{\ad}\st{\sim}\r S_{\bullet}(\mod{k})$
fitting into a commutative diagram
$$\xymatrix@C=25pt{\mathcal{B}_{\bullet}(k)^{\ad}\ar[rr]^-{\sim}\ar@{^{(}->}[d]&& S_{\bullet}(\mod{k})\ar[d]_{\text{induced by }q}\\
\mathcal{E}_{\bullet}(k)^{\ad}\ar@{^{(}->}[d]&& S_{\bullet}(\mod{R})\ar[d]_{\text{canonical}}
\ar[rr]^-{
\begin{array}{c}
\text{\scriptsize complexes}\vspace{-3pt}\\ \text{\scriptsize  concentrated}\vspace{-3pt}\\ \text{\scriptsize  in degree $0$}
\end{array}}& 
&\ho(S_{\bullet}(C^{b}(\mod{R})))\ar[d]\\
\mathcal{D}_{\bullet}(k)^{\ad}\ar[rr]^-{\sim}_-{\tilde{\xi}(R,\alpha)}&& \ho(S_{\bullet}(\C{W}_{R}))\ar[rr]^-{\sim}_{\text{Proposition \ref{auxiliary2}}}&&\ho(S_{\bullet}(\tilde{\C{W}}_{R}))}$$
Considering the subcategories of isomorphisms and taking the loop space of the geometric realizations in this diagram, we obtain the following commutative diagram,
$$\xymatrix@C=45pt{\Omega|\iso{\mathcal{B}_{\bullet}(k)^{\ad}}|\ar[r]^-{\sim}\ar@{^{(}->}[dd]& K(\mod{k})\ar[d]^{\text{d\'evissage}}_{\sim}\\
& K(\mod{R})\ar[d]
\ar[r]^-{\rho}& K(\der(\mod{R}))\ar[d]\\
\Omega|\iso{\mathcal{D}_{\bullet}(k)^{\ad}}|\ar[r]^-{\sim}& K(\der(\C{W}_{R}))\ar[r]^-{\sim}_-{\text{Props. \ref{auxiliary1} and \ref{auxiliary2}}}&K(\underline{\der}(R))}$$
Assume that derivator $K$-theory satisfies both the localization and comparison conjectures. Then $\rho$ is an equivalence and the short exact sequence of triangulated derivators $$\der(\proj{R}) \To \der(\mod{R}) \To \underline{\der}(R)$$ induces a homotopy fibration of $K$-theory spaces
\begin{equation*}
K(\der(\proj{R})) \To K(\der(\mod{R})) \To K(\underline{\der}(R)),
\end{equation*}
therefore $K(R)=K(\proj{R})\st{\rho}\simeq K(\der(\proj{R}))$ is the homotopy fiber of the map $$\Omega|\iso{\mathcal{B}_{\bullet}(k)^{\ad}}|\rightarrow \Omega|\iso{\mathcal{D}_{\bullet}(k)^{\ad}}|,$$ which is independent of $R$ (it only depends on $k$), but this contradicts the fact that, for any prime number $p$,  $\mathbb{F}_{p}[\varepsilon]/\varepsilon^{2}$ and $\mathbb{Z}/p^{2}$ have inequivalent Quillen $K$-theories. More precisely, if $p>3$,
\begin{align*}
K_3(\mathbb{F}_{p}[\varepsilon]/\varepsilon^{2})&\cong \mathbb{Z}/p\oplus \mathbb{Z}/p \oplus\mathbb{Z}/(p^{2}-1)\\&\ncong  \mathbb{Z}/p^{2} \oplus\mathbb{Z}/(p^{2}-1)\cong K_3(\mathbb{Z}/p^{2}),& \text{\cite{allps, k3wvff}},\\
 K_4(\mathbb{F}_{3}[\varepsilon]/\varepsilon^{2})\cong 0 &\ncong \mathbb{Z}/3\cong K_4(\mathbb{Z}/9),& \text{\cite{otaktozpn}},\\ 
 K_2(\mathbb{F}_{2}[\varepsilon]/\varepsilon^{2})\cong0 &\ncong\mathbb{Z}/2\cong K_2(\mathbb{Z}/4),& \text{\cite{k2nd, k2dvr}}.
\end{align*}
\end{proof}

\begin{rem}Although the suspension functor of the triangulated category $\ho(\C{W}_{R})$ can be easily seen to be naturally isomorphic to the identity functor (e.g. see \cite{kttc}), this is not true in general 
for the triangulated categories $\ho(\C{W}_{R}^{\deltan{n}})$. For example, the suspension of the morphism in $\ho(\C{W}_{R}^{\deltan{1}})$ defined by the square
$$\xymatrix{
k \ar[r]^{\scriptstyle \alpha} \ar[d] &R \ar[d]^{\scriptstyle q} \\
0 \ar[r] & k}
$$
is isomorphic to the morphism 
$$\xymatrix{
k \ar[r] \ar[d]_{\scriptstyle \alpha} & 0 \ar[d] \\
R \ar[r]^{\scriptstyle q} & k }
$$
These are isomorphic to the morphisms 
$$\xymatrix{
k \ar[r]^{\scriptstyle \alpha} \ar[d]_-{\scriptstyle 0} & R \ar[d]^{\scriptstyle q} \\
R \ar[r]^{\scriptstyle q} & k}
$$
and 
$$\xymatrix{
k \ar[r]^{\scriptstyle \alpha} \ar[d]_-{\scriptstyle \alpha} &R \ar[d]^{\scriptstyle 0} \\
R \ar[r]^{\scriptstyle q} & k}
$$
respectively. Then note that the sum of the two morphisms represents the zero morphism from $k \st{\scriptstyle \alpha}{\to} R$ to $R \st{\scriptstyle q}{\to} k$ in $\ho(\C{W}_{R}^{\deltan{1}})$. Thus the argument 
of \cite[2.3]{kttc} does not apply to show that the additivity theorem for derivator $K$-theory \cite{adkt} and the agreement on $K_1$ \cite{malt} contradict the localization conjecture.
\end{rem}




\begin{thebibliography}{ALPS85}

\bibitem[ALPS85]{allps}
J.~E. Aisbett, E.~Lluis-Puebla, and V.~Snaith, \emph{On {$K_\ast({\bf Z}/n)$}
  and {$K_\ast({\bf F}_q[t]/(t^2))$}}, Mem. Amer. Math. Soc. \textbf{57}
  (1985), no.~329, vi+200.

\bibitem[Ang]{otaktozpn}
V.~Angeltveit, \emph{On the algebraic {$K$}-theory of {$\mathbb{Z}/p^n$}}, in
  preparation.

\bibitem[ARS95]{rtaa}
M.~Auslander, I.~Reiten, and S.~O. Smal{\o}, \emph{Representation theory of
  {A}rtin algebras}, Cambridge Studies in Advanced Mathematics, vol.~36,
  Cambridge University Press, Cambridge, 1995.

\bibitem[BG82]{csrt}
K.~Bongartz and P.~Gabriel, \emph{Covering spaces in representation-theory},
  Invent. Math. \textbf{65} (1981/82), no.~3, 331--378.

\bibitem[Cis02]{tckt}
D.-C. Cisinski, \emph{Th{\'e}or{\`e}mes de cofinalit{\'e} en ${K}$-th{\'e}orie
  (d'apr{\`e}s {T}homason)}, seminar notes,
  \texttt{http://www.math.univ-paris13.fr/\~{}cisinski/cofdev.dvi}, 2002.

\bibitem[Cis10a]{ciscd}
\bysame, \emph{Cat{\'e}gories d{\'e}rivables}, Bull. Soc. Math. France
  \textbf{138} (2010), no.~3, 317--393.

\bibitem[Cis10b]{ikted}
\bysame, \emph{Invariance de la {$K$}-th\'eorie par \'equivalences
  d\'eriv\'ees}, J. {$K$}-Theory \textbf{6} (2010), no.~3, 505--546.

\bibitem[CN08]{adkt}
D.-C. Cisinski and A.~Neeman, \emph{Additivity for derivator ${K}$-theory},
  Adv. Math. \textbf{217} (2008), no.~4, 1381--1475.

\bibitem[DS75]{k2dvr}
R.~K. Dennis and M.~R. Stein, \emph{{$K\_{2}$} of discrete valuation rings},
  Advances in Math. \textbf{18} (1975), no.~2, 182--238.

\bibitem[DS04]{ktde}
D.~Dugger and B.~Shipley, \emph{{$K$}-theory and derived equivalences}, Duke
  Math. J. \textbf{124} (2004), no.~3, 587--617.

\bibitem[DS09]{curious}
\bysame, \emph{A curious example of triangulated-equivalent model categories
  which are not {Q}uillen equivalent}, Algebr. Geom. Topol. \textbf{9} (2009),
  no.~1, 135--166.

\bibitem[Fra96]{frankepre}
J.~Franke, \emph{Uniqueness theorems for certain triangulated categories
  possessing an {A}dams spectral sequence}, {$K$}-theory Preprint Archives no.
  139, \texttt{http://www.math.uiuc.edu/K-theory/}, 1996.

\bibitem[Gar05]{sdckt2}
G.~Garkusha, \emph{Systems of diagram categories and ${K}$-theory {II}}, Math.
  Z. \textbf{249} (2005), no.~3, 641--682.

\bibitem[Gar06]{sdckt1}
\bysame, \emph{Systems of diagram categories and ${K}$-theory {I}}, Algebra i
  Analiz \textbf{18} (2006), no.~6, 131--186, (Russian).

\bibitem[Gei97]{k3wvff}
T.~Geisser, \emph{On {$K_3$} of {W}itt vectors of length two over finite
  fields}, $K$-Theory \textbf{12} (1997), no.~3, 193--226.

\bibitem[Gil81]{rrthakt}
H.~Gillet, \emph{{R}iemann-{R}och theorems for higher algebraic ${K}$-theory},
  Adv. Math. \textbf{40} (1981), 203--289.

\bibitem[Gro90]{derivateurs}
A.~Grothendieck, \emph{Les d{\'e}rivateurs}, manuscript, \newline
  \texttt{http://www.math.jussieu.fr/\~{}maltsin/groth/Derivateurs.html}, 1990.

\bibitem[Hov99]{hmc}
M.~Hovey, \emph{Model categories}, Mathematical Surveys and Monographs,
  vol.~63, American Mathematical Society, Providence, RI, 1999.

\bibitem[Kel07]{dtce}
B.~Keller, \emph{Appendice: {L}e d{\'e}rivateur triangul{\'e} associ{\'e} {\`a}
  une cat{\'e}gorie exacte}, Categories in algebra, geometry and mathematical
  physics, Contemp. Math., vol. 431, Amer. Math. Soc., Providence, RI, 2007,
  pp.~369--373.

\bibitem[Lam99]{lmr}
T.~Y. Lam, \emph{Lectures on modules and rings}, Graduate Texts in Mathematics,
  vol. 189, Springer-Verlag, New York, 1999.

\bibitem[Mal07]{ktdt}
G.~Maltsiniotis, \emph{La {$K$}-th\'eorie d'un d\'erivateur triangul\'e},
  Categories in algebra, geometry and mathematical physics, Contemp. Math.,
  vol. 431, Amer. Math. Soc., Providence, RI, 2007, pp.~341--368.

\bibitem[Mit72]{rso}
B.~Mitchell, \emph{Rings with several objects}, Advances in Mathematics
  \textbf{8} (1972), 1--161.

\bibitem[Mur08]{malt}
F.~Muro, \emph{Maltsiniotis's first conjecture for ${K}\_1$}, Int. Math. Res.
  Notices \textbf{2008} (2008), rnm153--31.

\bibitem[Qui67]{ha}
D.~Quillen, \emph{Homotopical algebra}, Lecture Notes in Mathematics, vol.~43,
  Springer-Verlag, Berlin, 1967.

\bibitem[Qui72]{otcktglgoff}
\bysame, \emph{On the cohomology and {$K$}-theory of the general linear groups
  over a finite field}, Ann. of Math. (2) \textbf{96} (1972), 552--586.

\bibitem[Ren09]{pcthqd}
O.~Renaudin, \emph{Plongement de certaines th{\'e}ories homotopiques de
  {Q}uillen dans les d{\'e}rivateurs}, J. Pure Appl. Algebra \textbf{213}
  (2009), no.~10, 1916--1935.

\bibitem[Ric89]{dcse}
J.~Rickard, \emph{Derived categories and stable equivalence}, J. Pure Appl.
  Algebra \textbf{61} (1989), no.~3, 303--317.

\bibitem[Sch02]{kttc}
M.~Schlichting, \emph{A note on {$K$}-theory and triangulated categories},
  Invent. Math. \textbf{150} (2002), no.~1, 111--116.

\bibitem[TT90]{tt}
R.~W. Thomason and T.~Trobaugh, \emph{Higher algebraic ${K}$-theory of schemes
  and derived categories}, The Grothendieck festschrift, a collection of
  articles written in honor pf the 60th birthday of Alexander Grothendieck,
  Progress in Math. 88, vol. III, Brikh{\"a}user, Boston, 1990, pp.~247--435.

\bibitem[TV04]{ktsc}
B.~To{\"e}n and G.~Vezzosi, \emph{A remark on {$K$}-theory and
  {$S$}-categories}, Topology \textbf{43} (2004), no.~4, 765--791.

\bibitem[vdK71]{k2nd}
W.~van~der Kallen, \emph{Le {$K\_{2}$} des nombres duaux}, C. R. Acad. Sci.
  Paris S{\'e}r. A-B \textbf{273} (1971), A1204--A1207.

\bibitem[Wal85]{akts}
F.~Waldhausen, \emph{Algebraic {$K$}-theory of spaces}, Algebraic and geometric
  topology (New Brunswick, N.J., 1983), Lecture Notes in Math., vol. 1126,
  Springer, Berlin, 1985, pp.~318--419.

\end{thebibliography}

\providecommand{\bysame}{\leavevmode\hbox to3em{\hrulefill}\thinspace}
\providecommand{\MR}{\relax\ifhmode\unskip\space\fi MR }
\providecommand{\MRhref}[2]{%
  \href{http://www.ams.org/mathscinet-getitem?mr=#1}{#2}
}
\providecommand{\href}[2]{#2}

\end{document}